\documentclass[a4paper]{article}
\usepackage[utf8]{inputenc}
\usepackage[english]{babel}
\usepackage{amsmath}
\usepackage{amsthm}
\usepackage{amssymb}
\usepackage{bbm}	%% used for \mathbbm
\usepackage[numbers]{natbib}	%% bibliography
\usepackage{graphicx, color, soul} %% soul for highlight
\usepackage[title]{appendix}
\usepackage[paperwidth=19.5cm,paperheight=25cm]{geometry}
\usepackage{url}
\usepackage{blkarray} %matrix with labels
\usepackage[colorlinks=true,linkcolor=blue,citecolor=blue]{hyperref}
\usepackage[capitalize,nameinlink]{cleveref}

\usepackage{authblk} %% authors and affiliation

\usepackage[shortlabels]{enumitem}	%% change style of lists as "[(i)]"
\setlist{labelsep=.25in,leftmargin=*,labelindent=1cm,topsep=2pt,noitemsep}	%% adjust lists
\setlist[enumerate]{label=(\roman*)}

\theoremstyle{plain}
\newtheorem{theorem}{Theorem}

\newtheorem{proposition}[theorem]{Proposition}
\newtheorem{lemma}[theorem]{Lemma}

\theoremstyle{remark}
\newtheorem*{remark}{Remark}
\theoremstyle{definition}
\newtheorem{example}[theorem]{Example}

\newcommand{\Polya}{P\'{o}lya }

\usepackage{geometry}
\title{On some properties of probability kernels}
\author{Hristo Sariev\thanks{h.sariev@math.bas.bg}}
\affil{\normalsize Institute of Mathematics and Informatics, Bulgarian Academy of Sciences, Sofia, Bulgaria\\ \vspace{0.1cm}\normalsize Faculty of Mathematics and Informatics, Sofia University ``St. Kliment Ohridski'', Sofia, Bulgaria}

\date{}

\begin{document}

\maketitle

\begin{abstract}
Although regular conditional distributions (r.c.d.) are well-defined and widely used measure-theoretic objects, they can violate our intuition from the classical definition of a conditional probability given an event. For that purpose, the notion of a proper r.c.d. has been introduced. Here, we study how properness, viewed as a property of probability kernels in general, is related to stationarity, compatibility, reversibility and totality, revealing the effects these properties have on the structure of probability kernels. As a further development, we consider the inverse problem of characterizing certain classes of r.c.d.s in terms of the above properties. In particular, we derive necessary and sufficient conditions under which, for a given probability kernel, there exists a unique (in some sense) sub-$\sigma$-algebra such that the probability kernel is a proper r.c.d. given that sub-$\sigma$-algebra.
% At the end, we return to the problem of proper r.c.d.s.
% In the process, we reveal a ``block-diagonal'' design that is inherent in any r.c.d. 
\end{abstract}
\noindent{\bf Keywords:}
Regular conditional distribution, proper conditional distribution, probability kernel, countably generated $\sigma$-algebra

\noindent{\bf MSC2020 Classification:} 60A10, 60A05, 28A99

\section{Introduction}\label{section:introduction}

Let $(\mathbb{X},\mathcal{X})$ be a standard Borel space, $\nu$ a probability measure on $\mathbb{X}$, and $\mathcal{G}$ a sub-$\sigma$-algebra of $\mathcal{X}$ on $\mathbb{X}$. Denote by $\nu_{\mathcal{G}}$ the restriction of $\nu$ on $\mathcal{G}$. We will say that $\mathcal{G}$ is \textit{countably generated (c.g.) under $\nu$} if there exists $C\in\mathcal{G}$ such that $\nu(C)=1$ and $\mathcal{G}\cap C$ is c.g. For each $x\in\mathbb{X}$, define the \textit{$\mathcal{G}$-atom} of $x$ to be
\[[x]_\mathcal{G}:=\bigcap_{x\in G\in\mathcal{G}}G.\]
In general, atoms need not be measurable sets, but for c.g. $\sigma$-algebras this is true: if $\mathcal{F}=\sigma(F_1,F_2,\ldots)$, for some $F_1,F_2,\ldots\subseteq\mathbb{X}$, then
\[[x]_\mathcal{F}=\{y\in\mathbb{X}:\delta_x(F)=\delta_y(F)\mbox{ for all }F\in\mathcal{F}\}=\{y\in\mathbb{X}:\delta_x(F_n)=\delta_y(F_n)\mbox{ for all }n\geq1\}\in\mathcal{F}.\]
If $\mathcal{F}$ is just c.g. under $\nu$, then $[x]_\mathcal{F}\in\mathcal{F}$ for $\nu$-a.e. $x$.
%% $\mathcal{F}':=\{(G\cap C)\cup F:G\in\mathcal{G},F=\emptyset\mbox{ or }F=C^c\}$ is an atomic sub-$\sigma$-algebra$ of $\mathcal{F}$, such that $[x]_{\mathcal{F}'}=[x]_\mathcal{F}$, for $x\in C$. 
%% $\{F_1,F_2,\ldots\}$ does not need to be a $\pi$-class
%Fix $x\in\mathbb{X}$. Let $y\in[x]_\mathcal{G}$ and $G\in\mathcal{G}$. If $x\in G$, then $y\in G$, otherwise if $x\in G^c$, then $y\in G^c$; thus, $\delta_x(G)=\delta_y(G)$. Conversely, let $y\in\mathbb{X}$ be such that $\delta_x(G)=\delta_y(G)$ for all $G\in\mathcal{G}$. Then $y\in G$ whenever $x\in G\in\mathcal{G}$.
%In other words, $[x]_\mathcal{G}=\bigcap_{n=1}^\infty G_n^{\delta_x(G_n)}$, using the notation $G_n^1=G$ and $G_n^0=G_n$.

A probability kernel on $\mathbb{X}$ is a function $R:\mathbb{X}\times\mathcal{X}\rightarrow[0,1]$ that satisfies $(i)$ the map $x\mapsto R(x,B)\equiv R_x(B)$ from $\mathbb{X}$ to $[0,1]$ is $\mathcal{X}$-measurable, for all $B\in\mathcal{X}$; and $(ii)$ $B\mapsto R_x(B)$ is a probability measure on $\mathbb{X}$, for all $x\in\mathbb{X}$. Moreover, a probability kernel $R$ on $\mathbb{X}$ is said to be a \textit{regular version of the conditional distribution} (r.c.d.) for $\nu$ given $\mathcal{G}$, denoted by
\[R(\cdot)=\nu(\,\cdot\mid\mathcal{G}),\]
if, in addition to $(i)$-$(ii)$, it holds $(iii)$ $x\mapsto R_x(B)$ is $\mathcal{G}$-measurable, for all $B\in\mathcal{X}$; and $(iv)$ $\int_AR_x(B)\nu(dx)=\nu(A\cap B)$, for all $A\in\mathcal{G}$ and $B\in\mathcal{X}$. The assumptions on $(\mathbb{X},\mathcal{X})$ guarantee that an r.c.d. for $\nu$ given $\mathcal{G}$ exists and is unique up to a $\nu$-null set.

From both a foundational and a technical point of view, it is desirable to work with probability kernels $R$ on $\mathbb{X}$ that are
%%In the literature, it is common to work with probability kernels $R$ on $\mathbb{X}$, which are
\begin{enumerate}
\item[\textsf{(S)}] \textit{stationary} with respect to (w.r.t.) $\nu$, provided as measures on $\mathbb{X}$,
\[\int_\mathbb{X}R_x(dy)\nu(dx)=\nu(dy);\]
\item[\textsf{(SC)}] \textit{self-compatible}, provided as measures on $\mathbb{X}$,
\[\int_\mathbb{X}R_y(dz)R_x(dy)=R_x(dz)\qquad\mbox{for }\nu\mbox{-almost every (a.e.) }x;\]
\item[\textsf{(R)}] \textit{reversible} w.r.t. $\nu$, provided as measures on $\mathbb{X}^2$,
\[R_x(dy)\nu(dx)=R_y(dx)\nu(dy);\]
\item[\textsf{(SR)}] \textit{self-reversible}, provided as measures on $\mathbb{X}^2$,
\[R_y(dz)R_x(dy)=R_z(dy)R_x(dz)\qquad\mbox{for }\nu\mbox{-a.e. }x;\]
\item[\textsf{(P)}] \textit{a.e. proper} w.r.t. $\mathcal{G}$, provided there exists $F\in\mathcal{G}$ such that $\nu(F)=1$ and
\[R_x(A)=\delta_x(A),\qquad\mbox{for all }A\in\mathcal{G}\mbox{ and }x\in F;\]
\item[\textsf{(T)}] \textit{a.e. total} on the atoms of $\mathcal{G}$, provided there exists $G\in\mathcal{G}$ such that $\nu(G)=1$ and, for all $x\in G$, it holds $[x]_\mathcal{G}\in\mathcal{G}$ and
\[R_x([x]_\mathcal{G})=1.\]
\end{enumerate}

Trivially, \textsf{(R)} implies \textsf{(S)}, and \textsf{(SR)} implies \textsf{(SC)}. Furthermore, \textsf{(T)} follows from \textsf{(P)} when $\mathcal{G}$ contains $\nu$-a.e. of its atoms. The converse result is easily established. Letting $R$ satisfy \textsf{(T)} on $G\in\mathcal{G}$ with $\nu(G)=1$, we get, for all $A\in\mathcal{G}$ and $x\in G$, that $R_x(A)\geq R_x([x]_\mathcal{G})=1$ when $x\in A$, so $R_x(A)=1=\delta_x(A)$; otherwise, if $x\in A^c$, we have again $R_x(A)=1-R_x(A^c)=0=\delta_x(A)$.
%% Moreover, (SC), (T) and $\sigma(R)\subseteq\mathcal{G}$ imply that $R$ is $\nu$-a.e. constant on $\mathcal{G}$-atoms, and thus, by \cite[Lemma 5.18]{einsielder2011}, $x\mapsto R_x(B)$ is measurable w.r.t. the completion of $\mathcal{G}$, for all $B\in\mathcal{X}$.

Now, assume that $R(\cdot)=\nu(\,\cdot\mid\mathcal{G})$ is an r.c.d. for $\nu$ given some sub-$\sigma$-algebra $\mathcal{G}\subseteq\mathcal{X}$. In that case, \textsf{(S)}, \textsf{(SC)}, \textsf{(R)}, \textsf{(SR)} follow from standard results on conditional distributions. Indeed, for every $A,B\in\mathcal{X}$,
\begin{equation}\label{eq:introduction:rcd_reverse1}
\int_A\nu(B|\mathcal{G})(x)\nu(dx)=\mathbb{E}_\nu[\mathbbm{1}_A\cdot\nu(B|\mathcal{G})]=\mathbb{E}_\nu[\nu(A|\mathcal{G})\nu(B|\mathcal{G})]=\int_B\nu(A|\mathcal{G})(x)\nu(dx),
\end{equation}
and, for $\nu$-a.e. $x$,
\[\int_A\nu(B|\mathcal{G})(y)\nu(dy|\mathcal{G})(x)=\mathbb{E}_\nu[\mathbbm{1}_A\cdot\nu(B|\mathcal{G})|\mathcal{G}](x)=\nu(A|\mathcal{G})(x)\nu(B|\mathcal{G})(x)=\int_B\nu(A|\mathcal{G})(y)\nu(dy|\mathcal{G})(x).\]
As for \textsf{(P)}, arguing through \eqref{eq:introduction:rcd_reverse1}, we get
\[R_x(A)=\delta_x(A),\qquad\mbox{for all }A\in\mathcal{G}\mbox{ and }\nu\mbox{-a.e. }x,\]
where the essential set may depend on $A$. Thus, when $\mathcal{G}$ is c.g., combining these essential sets leads to \textsf{(P)}. In fact, as shown by Proposition \ref{result:introduction:proper:characterization}, for r.c.d.s, conditions \textsf{(P)} and \textsf{(T)} ultimately depend on the properties of the conditioning $\sigma$-algebra $\mathcal{G}$, see also \cite[p.\,650]{berti2007} and \cite[Theorem 1]{blackwell1975}.
%Moreover, Proposition \ref{result:introduction:proper:characterization} puts into perspective the role of the the class of $\sigma$-algebras that are c.g. under $\nu$.

\begin{proposition}\label{result:introduction:proper:characterization}
The following conditions are equivalent:
\begin{enumerate}
\item $\mathcal{G}$ is c.g. under $\nu$;
\item $\nu(\,\cdot\mid\mathcal{G})$ satisfies \textnormal{\textsf{(P)}};
\item $\nu(\,\cdot\mid\mathcal{G})$ satisfies \textnormal{\textsf{(T)}}.
\end{enumerate}
\end{proposition}
\begin{proof} \noindent\textit{$(i)\Rightarrow(ii)$.}
Let $C\in\mathcal{G}$ be such that $\nu(C)=1$ and $\mathcal{G}\cap C=\sigma(B_1,B_2,\ldots)$, where $\{B_1,B_2,\ldots\}$ forms a $\pi$-class on $C$, without loss of generality. Then $B_n=B_n^*\cap C$ for some $B_n^*\in\mathcal{G}$. It follows for each $n=1,2,\ldots$ and every $A\in\mathcal{G}$ that
\[\int_A\nu(B_n^*|\mathcal{G})(x)\nu(dx)=\int_A\delta_x(B_n^*)\nu(dx);\]
thus, $\nu(B_n^*|\mathcal{G})(x)=\delta_x(B_n^*)$ for $\nu$-a.e. $x$. Combining the essential sets, we obtain $C_1\in\mathcal{G}$ such that $\nu(C_1)=1$ and, for every $x\in C_1$, it holds $\nu(B_n^*|\mathcal{G})(x)=\delta_x(B_n^*)$ for all $n\geq1$. On the other hand, $\nu(C)=1$ implies that $\nu(C|\mathcal{G})(x)=1$ for $\nu$-a.e. $x$, say on $C_2\in\mathcal{G}$ with $\nu(C_2)=1$. Then $\nu(B_n|\mathcal{G})(x)=\delta_x(B_n)$ and, ultimately, from the determining class theorem, $\nu(B|\mathcal{G})(x)=\delta_x(B)$, for all $B\in\mathcal{G}\cap C$ and $x\in C\cap C_1\cap C_2$. Therefore,
\[\nu(A|\mathcal{G})(x)=\nu(A\cap C|\mathcal{G})(x)=\delta_x(A\cap C)=\delta_x(A),\qquad\mbox{for all }A\in\mathcal{G}\mbox{ and }x\in C\cap C_1\cap C_2.\]

\noindent\textit{$(ii)\Rightarrow(i)$.}
Suppose that $\nu(\,\cdot\mid\mathcal{G})$ is a.e. proper on $F\in\mathcal{G}$, where $\nu(F)=1$. By hypothesis, $\mathcal{X}=\sigma(E_1,E_2,\ldots)$, for some $\pi$-class $\{E_1,E_2,\ldots\}$ on $\mathbb{X}$. Let us define
\[\mathcal{G}^*:=\sigma(x\mapsto\nu(E_n|\mathcal{G})(x),n\geq1)\qquad\mbox{and}\qquad\mathcal{F}:=\{A\in\mathcal{X}:\nu(A|\mathcal{G})\mbox{ is }\mathcal{G}^*\mbox{-measurable}\}.\]
Then $\mathcal{G}^*$ is c.g. and $\mathcal{G}^*\subseteq\mathcal{G}$. Moreover, $\mathcal{F}$ is a $\lambda$-class such that $\mathcal{F}\supseteq\{E_1,E_2,\ldots\}$, so $\mathcal{F}=\mathcal{X}$ by Dynkin's lemma, and thus $\mathcal{G}^*=\sigma(x\mapsto\nu(B|\mathcal{G})(x),B\in\mathcal{X})$. Let $A\in\mathcal{G}$. Since $\nu(A|\mathcal{G})(x)=\delta_x(A)$ for $x\in F$, we have $A\cap F\in\mathcal{G}^*\cap F$; therefore, $\mathcal{G}^*\cap F=\mathcal{G}\cap F$, which implies that $\mathcal{G}\cap F$ is c.g.\\

\noindent\textit{$(ii)\Leftrightarrow(iii)$.}
Suppose that $\nu(\,\cdot\mid\mathcal{G})$ is a.e. proper on $F\in\mathcal{G}$, where $\nu(F)=1$. As $(ii)$ implies $(i)$, there exists $C\in\mathcal{G}$ such that $\nu(C)=1$ and $\mathcal{G}\cap C$ is c.g. Then $\mathcal{G}\cap C$ contains all of its atoms, so it follows for every $x\in C$ that $[x]_\mathcal{G}=[x]_\mathcal{G}\cap C=[x]_{\mathcal{G}\cap C}\in\mathcal{G}\cap C\subseteq\mathcal{G}$. Therefore, $\nu([x]_\mathcal{G}|\mathcal{G})(x)=\delta_x([x]_\mathcal{G})=1$, for $x\in C\cap F$. The converse result was already established.
%Define $\mathcal{G}':=\{(G\cap C)\cup F:G\in\mathcal{G},F=\emptyset\mbox{ or }F=C^c\}$. It is not hard to see that $\mathcal{G}'$ is a c.g. sub-$\sigma$-algebra of $\mathcal{G}$ on $\mathbb{X}$ such that $[x]_{\mathcal{G}}=[x]_{\mathcal{G}'}\in\mathcal{G}'\subseteq\mathcal{G}$, for all $x\in C$.
\end{proof}

In the literature, a.e. total r.c.d.s are used in the derivation of some results in ergodic theory, see, e.g., \cite[Theorem 6.2, Theorem 7.21]{einsielder2011}. In contrast, \cite{berti2007,seidenfeld2001} are devoted to the study of r.c.d.s which are explicitly improper, and hence do not satisfy \textsf{(P)} and \textsf{(T)}, but appear in some familiar contexts. We discuss the properness of r.c.d.s in more detail in Section \ref{section:proper}. On the other hand, the conditions listed above are assumed in many situations where $R$ is not necessarily an r.c.d. For example, \textsf{(S)} and \textsf{(R)} are related to the stationarity and reversibility of Markov chains, respectively, where $R$ is the transition kernel of a Markov chain, see, e.g., \cite[Example 4]{berti2014}. In turn, \textsf{(T)} is a foundational property of disintegrations of probability measures, in the sense of \citep{berti1999}. The pair of reversibility conditions \textsf{(R)} and \textsf{(SR)} are themselves central in the analysis of exchangeable measure-valued \Polya urn processes with reinforcement kernel $R$ and initial composition $\nu$, see \cite[proof of Theorem 3.7]{sariev2024}. In the same setting, \textsf{(P)} is explicitly used to get a characterization of these models via predictive sufficiency, see \cite[Theorem 4]{sariev2025}. In the case of a system of probability kernels, compatibility-type conditions like \textsf{(SC)} are assumed in the context of Gibbs measures from statistical mechanics, see, e.g., \cite[Section 2]{preston1976} and \cite[p.\,538]{sokal1981}, and are the basis of some predictive construction of conditionally identically distributed sequences of random variables \cite{berti2021}; see also \cite{berti2013}.

Our goal is to study how these conditions are related to each other, beyond the more obvious relations mentioned above, and to determine the effects they have on $R$. These results are contained in Section \ref{section:connections}. As a further development, in Section \ref{section:charcterization}, we consider the inverse problem of characterizing certain classes of r.c.d.s by some combination of the properties stationarity, reversibility and totality. In particular, we derive  necessary and sufficient conditions for a probability kernel to be an r.c.d. for $\nu$ given some sub-$\sigma$-algebra, which is c.g. under $\nu$ or generated by a countable partition under $\nu$. In Section \ref{section:proper}, we return to the problem of improper r.c.d.s.

\section{Relations}\label{section:connections}

Let $R$ be a probability kernel on $\mathbb{X}$, and define the generated $\sigma$-algebra of $R$ by
\[\sigma(R):=\sigma(x\mapsto R_x(B),B\in\mathcal{X}).\]
Following the proof of Proposition \ref{result:introduction:proper:characterization}, it can be shown that $\sigma(R)$ is c.g. Moreover, the $\sigma(R)$-atom of each $x$ coincides with the set of all $y\in\mathbb{X}$ such that, as measures on $\mathbb{X}$, it holds $R_y(dz)=R_x(dz)$, see Lemma \ref{lemma:connections:atoms_R}. Thus, for $\mathcal{G}=\sigma(R)$, assumption \textsf{(T)} becomes analogous to the classical property that we can pull out known factors from conditional expectations.

\begin{lemma}\label{lemma:connections:atoms_R}
For each $x\in\mathbb{X}$,
\[[x]_{\sigma(R)}=\{y\in\mathbb{X}:R_y(dz)=R_x(dz)\}.\]
Moreover, $x\mapsto R_x([x]_{\sigma(R)})$ is $\sigma(R)$-measurable.
%% Mentioned but not proven in Lemma 10 in Berti and Rigo (2007).
\end{lemma}
\begin{proof}
By hypothesis, $\mathcal{X}=\sigma(E_1,E_2,\ldots)$, for some $\pi$-class $\{E_1,E_2,\ldots\}$ on $\mathbb{X}$. Let us define
\[D:=\bigl\{(x,y)\in\mathbb{X}^2:R_x(dz)=R_y(dz)\bigr\},\]
and denote by $D_x$ the $x$-section of $D$. It follows from standard results that $D=\{(x,y)\in\mathbb{X}^2:R_x(E_n)=R_y(E_n),n\geq1\}$. Since $(x,y)\mapsto(R_x(E_n),R_y(E_n))$ is $\sigma(R)\otimes\sigma(R)$-measurable, for all $n\geq1$, and the diagonal of $[0,1]^2$ is a measurable set, we get $D\in\sigma(R)\otimes\sigma(R)$, so in particular $D_x\in\sigma(R)$, for all $x\in\mathbb{X}$. Fix $x\in\mathbb{X}$. As $x\in D_x$, we have $[x]_{\sigma(R)}\subseteq D_x$. Take $y\in D_x$. It follows that $R_y(dz)=R_x(dz)$, so for any $H\in\sigma(R)$, it holds $x\in H$ if and only if $y\in H$; thus, $[x]_{\sigma(R)}=[y]_{\sigma(R)}$. Then $y\in[x]_{\sigma(R)}$, which implies that $[x]_{\sigma(R)}=D_x$.

Regarding the second statement, define
\[\mathcal{A}:=\Bigl\{E\in\sigma(R)\otimes\sigma(R):x\mapsto\int_{\mathbb{X}}\mathbbm{1}_E(x,y)R_x(dy)\mbox{ is }\sigma(R)\mbox{-measurable}\Bigr\}.\]
Let $A,B\in\sigma(R)$. It follows that $x\mapsto\int_{\mathbb{X}}\mathbbm{1}_{A\times B}(x,y)R_x(dy)=R_x(B)\delta_x(A)$ is $\sigma(R)$-measurable, so $A\times B\in\mathcal{A}$. In addition, it is easily seen that $\mathcal{A}$ is a $\lambda$-class, so Dynkin's lemma implies that $\mathcal{A}=\sigma(R)\otimes\sigma(R)$. Therefore, $D\in\mathcal{A}$ and $x\mapsto R_x([x]_{\sigma(R)})=\int_{\mathbb{X}}\mathbbm{1}_D(x,y)R_x(dy)$ is $\sigma(R)$-measurable.
\end{proof}
\begin{remark}
In Lemma \ref{lemma:connections:atoms_R}, suppose that $R$ satisfies \textsf{(SC)} everywhere on $\mathbb{X}$. Then
\[[x]_{\sigma(R)}=\{y\in\mathbb{X}:R_y(A)=R_x(A)\mbox{ for all }A\in\sigma(R)\},\]
that is, $\sigma(R)$-atoms are determined by $R$ on the generated $\sigma$-algebra. Indeed, let $y\in\mathbb{X}$ be such that, as measures on $(\mathbb{X},\sigma(R))$, it holds $R_y(dz)=R_x(dz)$. Then $R_x(B)=\int_\mathbb{X}R_z(B)R_x(dz)=\int_\mathbb{X}R_z(B)R_y(dz)=R_y(B)$, for all $B\in\mathcal{X}$. Since $\mathcal{X}$ is c.g., it follows that $y\in\{v\in\mathbb{X}:R_v(dz)=R_x(dz)\}=[x]_{\sigma(R)}$.
\end{remark}

If $\sigma(R)\subseteq\mathcal{G}$, for some sub-$\sigma$-algebra $\mathcal{G}\subseteq\mathcal{X}$, i.e. $x\mapsto R_x(B)$ is $\mathcal{G}$-measurable, for all $B\in\mathcal{G}$, then $[x]_\mathcal{G}\subseteq[x]_{\sigma(R)}$, for all $x\in\mathbb{X}$. In fact, by Lemma \ref{lemma:connections:constant} below, the latter is equivalent to $R$ being constant on the atoms of $\mathcal{G}$. Conversely, by Lemma 5.18 in \cite{einsielder2011}, if $\mathcal{G}$ is c.g. under $\nu$ and $R$ is $\nu$-a.e. constant on the atoms of $\mathcal{G}$, then $x\mapsto R_x(B)$ is measurable w.r.t. the completion of $\mathcal{G}$.

\begin{lemma}\label{lemma:connections:constant}
For each $x\in\mathbb{X}$,
\[[x]_\mathcal{G}\subseteq[x]_{\sigma(R)}\quad\mbox{ if and only if }\quad R_{x'}(dz)=R_x(dz)\mbox{ for all }x'\in[x]_\mathcal{G}.\]
\end{lemma}
\begin{proof}
If $[x]_\mathcal{G}\subseteq[x]_{\sigma(R)}$, then by Lemma \ref{lemma:connections:atoms_R}, it holds, for every $x'\in[x]_\mathcal{G}$, that $R_{x'}=R_x$. Conversely, take $x'\in[x]_\mathcal{G}$. By hypothesis, $R_{x'}=R_x$, so $x'\in[x]_{\sigma(R)}$, by Lemma \ref{lemma:connections:atoms_R}.
\end{proof}

The next statement relates the properties of stationarity, reversibility and totality. In particular, it shows that $R$ satisfies \textsf{(SR)} and \textsf{(R)} w.r.t. $\nu_{\sigma(R)}$ if and only if $R$ satisfies \textsf{(T)} w.r.t. $\sigma(R)$.

\begin{theorem}\label{result:connections:reverse<->total}
Let $R$ be a probability kernel on $\mathbb{X}$. Then
\begin{enumerate}
\item conditions \textnormal{\textsf{(SR)}} and \textnormal{\textsf{(S)}} w.r.t. $\nu_{\sigma(R)}$ imply \textnormal{\textsf{(T)}} w.r.t. $\sigma(R)$;
\item condition \textnormal{\textsf{(T)}} w.r.t. $\sigma(R)$ implies \textnormal{\textsf{(SR)}} and \textnormal{\textsf{(R)}} w.r.t. $\nu_{\sigma(R)}$.
\end{enumerate}
\end{theorem}
\begin{proof}
Suppose that $R$ satisfies \textsf{(SR)} and \textsf{(S)} w.r.t. $\nu_{\sigma(R)}$. Then there exists $C_0\in\sigma(R)$ such that $\nu(C_0)=1$ and, for all $x\in C_0$, as measures on $\mathbb{X}^2$,
%% $C_0\in\sigma(R)$ follows ultimately from $\mathcal{X}$ c.g.
\begin{equation}\label{proof:connections:reverse<->total:eq1}
R_y(dz)R_x(dy)=R_z(dy)R_x(dz),
\end{equation}
in the sense that $\int_{\mathbb{X}^2}f(y,z)R_y(dz)R_x(dy)=\int_{\mathbb{X}^2}f(z,y)R_y(dz)R_x(dy)$, for every non-negative $\mathcal{X}^2$-measurable function $f:\mathbb{X}^2\rightarrow\mathbb{R}_+$.

Define $C_n:=\{x\in C_{n-1}:R_x(C_{n-1})=1\}$, for $n\geq1$. Then $C_n\in\sigma(R)$ for all $n\geq1$. Since $R_x(\mathbb{X})=1$, we get from \textsf{(S)} w.r.t. $\nu_{\sigma(R)}$ that
\begingroup\allowdisplaybreaks
\begin{align*}
1=\nu(C_0)&=\int_\mathbb{X}R_x(C_0)\nu_{\sigma(R)}(dx)=1-\int_{\mathbb{X}}(1-R_x(C_0))\nu(dx)=1-\int_{C_1^c}(1-R_x(C_0))\nu(dx).
\end{align*}
\endgroup
But $R_x(C_0)<1$ for $x\in C_1^c$, so $\nu(C_1^c)=0$; otherwise, the term on the right-hand side of the equation becomes strictly less than $1$. Proceeding by induction, we get $\nu(C_n)=1$ for all $n\geq1$; thus, letting $C^*=\bigcap_{n=1}^\infty C_n$, it holds $C^*\in\sigma(R)$, $\nu(C^*)=1$, and $R_x(C_n)=1$, for all $x\in C^*$ and $n\geq1$. As a result, for all $x\in C^*$, we have that $R_x(C^*)=1$ and, since $C^*\subseteq C_0$, from \eqref{proof:connections:reverse<->total:eq1}, as measures on $\mathbb{X}^2$,
\begin{equation}\label{proof:connections:reverse<->total:eq2}
R_z(dv)R_y(dz)=R_v(dz)R_y(dv)\qquad\mbox{for }R_x\mbox{-a.e. }y.
\end{equation}
Applying \eqref{proof:connections:reverse<->total:eq1} and \eqref{proof:connections:reverse<->total:eq2} repeatedly, we obtain, for all $x\in C^*$, as measures on $\mathbb{X}^3$,
\begingroup\allowdisplaybreaks
\begin{align*}
R_z(dv)R_y(dz)R_x(dy)&=R_v(dz)R_y(dv)R_x(dy)=R_v(dz)R_v(dy)R_x(dv)\\
&=R_v(dy)R_z(dv)R_x(dz)=R_y(dv)R_z(dy)R_x(dz)=R_y(dv)R_y(dz)R_x(dy).
\end{align*}
\endgroup
Since $\mathcal{X}$ is c.g., we get, by Lemma \ref{lemma:connections:atoms_R},
\[R_y([y]_{\sigma(R)})=R_y(\{z\in\mathbb{X}:R_z(dv)=R_y(dv)\})=1\qquad\mbox{for }R_x\mbox{-a.e. }y,\]
that is, $R_x(\{y\in\mathbb{X}:R_y([y]_{\sigma(R)})=1\})=1$, for all $x\in C^*$. Moreover, $\{y\in\mathbb{X}:R_y([y]_{\sigma(R)})=1\}\in\sigma(R)$, so it follows from \textsf{(S)} w.r.t. $\nu_{\sigma(R)}$ that
\[\nu(\{y\in\mathbb{X}:R_y([y]_{\sigma(R)})=1\})=\int_{C^*}R_x(\{y\in\mathbb{X}:R_y([y]_{\sigma(R)})=1\})\nu_{\sigma(R)}(dx)=1.\]

Regarding the second statement, suppose that $R$ satisfies \textsf{(T)} w.r.t. $\sigma(R)$ on $G\in\sigma(R)$, where $\nu(G)=1$. By Lemma \ref{lemma:connections:atoms_R}, for every $x\in G$, we have, as measures on $\mathbb{X}$,
\[R_y(dz)=R_x(dz)\qquad\mbox{for }R_x\mbox{-a.e. }y.\]
It follows for every $A,B\in\mathcal{X}$ and all $x\in G$ that
\[\int_BR_y(A)R_x(dy)=\int_BR_x(A)R_x(dy)=R_x(A)R_x(B)=\int_AR_y(B)R_x(dy).\]
Finally, recall from Section \ref{section:introduction} that \textsf{(T)} implies \textsf{(P)} w.r.t. $\sigma(R)$. Then, for every $D,E\in\sigma(R)$, we have
\[\int_ER_x(D)\nu_{\sigma(R)}(dx)=\int_E\delta_x(D)\nu_{\sigma(R)}(dx)=\nu_{\sigma(R)}(E\cap D)=\int_DR_x(E)\nu_{\sigma(R)}(dx).\]
\end{proof}

The question naturally arises whether \textsf{(SR)} alone is equivalent to \textsf{(T)} w.r.t. $\sigma(R)$. As the next example demonstrates, the answer is negative. Nevertheless, it shows that when $\mathbb{X}$ is countable and $\nu(\{x\})>0$, for all $x\in\mathbb{X}$, then \textsf{(SR)} is enough to imply that $R$ is \textit{a.e. trivial} on the atoms of $\sigma(R)$, in the sense that
\[R_x([x]_{\sigma(R)})\in\{0,1\}\qquad\mbox{for }\nu\mbox{-a.e. }x.\]

\begin{example}\label{example:connections:countable}
Let $\mathbb{X}$ be countable, and $\nu(\{x\})>0$, for all $x\in\mathbb{X}$. In that case, $R$ is best described in terms of a matrix $[r_{xy}]_{x,y\in\mathbb{X}}$, where $r_{xy}=R_x(\{y\})$, so that \textsf{(SR)} becomes
\begin{equation}\label{eq:connections:SR->T}
r_{xy}r_{yz}=r_{xz}r_{zy}\qquad\mbox{for all }x,y,z\in\mathbb{X}.
\end{equation}
Fix $x\in\mathbb{X}$ such that $R_x([x]_{\sigma(R)})>0$. We can assume $r_{xx}>0$, without loss of generality, as there exists at least one $x'\in[x]_{\sigma(R)}$ such that $r_{x'x'}>0$ and $R_{x'}=R_x$. Let $y\notin[x]_{\sigma(R)}$ be such that $r_{xy}>0$. It follows from \eqref{eq:connections:SR->T} with $z=y$ that $r_{yx}=r_{xx}$. Take $z\in\mathbb{X}$. Then, $(i)$ if $z\in[x]_{\sigma(R)}$, we have $r_{xy}=r_{zy}>0$, so $r_{yz}=r_{xz}$ from \eqref{eq:connections:SR->T}; $(ii)$ if $z\notin[x]_{\sigma(R)}$ and $r_{xz}=0$, we have $r_{yz}=0=r_{xz}$ from \eqref{eq:connections:SR->T}; $(iii)$ if $z\notin[x]_{\sigma(R)}$ and $r_{xz}>0$, we have $r_{yx}r_{xz}=r_{yz}r_{zx}$ and $r_{yx}=r_{xx}=r_{zx}$, so $r_{yz}=r_{xz}$. As a result $R_y=R_x$, which implies from Lemma \ref{lemma:connections:atoms_R} that $y\in[x]_{\sigma(R)}$, absurd. Therefore, $r_{xy}=0$ and, ultimately, $R_x([x]_{\sigma(R)})=1$.

For $x\in\mathbb{X}$ such that $R_x([x]_{\sigma(R)})=0$ and $y\notin[x]_{\sigma(R)}$ with $r_{xy}>0$, we cannot exclude the possibility of $r_{yy}\neq r_{xy}$ to get a contradiction. For example, if $\mathbb{X}=\{1,2,3,4\}$, then
\[R=\left[\begin{array}{cccc} 0 & \frac{1}{3} & \frac{1}{3} & \frac{1}{3} \\ 0 & 1 & 0 & 0 \\ 0 & 0 & \frac{1}{2} & \frac{1}{2} \\ 0 & 0 & \frac{1}{2} & \frac{1}{2} \end{array}\right]\]
satisfies \textsf{(SR)}, and yet $R_1([1]_{\sigma(R)})=r_{11}=0$; thus, $R$ is trivial, but not total w.r.t. $\sigma(R)$.
\end{example}

On the other hand, Example 1 in \cite{berti2007} demonstrates that triviality is not an intrinsic property of r.c.d.s, despite the fact that they satisfies both \textsf{(R)} and \textsf{(SR)}.

\begin{example}[\cite{berti2007}, Example 1]\label{example:connections:improper}
Let $(\mathbb{X},\mathcal{X})=(\mathbb{R},\mathcal{B}(\mathbb{R}))$, and $\pi$ be a \textit{diffuse} probability measure on $\mathbb{X}$, i.e. $\pi(\{x\})=0$, for all $x\in\mathbb{X}$. Define
\[\nu:=\frac{1}{2}(\pi+\delta_0),\qquad\mathcal{G}:=\{B\in\mathcal{X}:\nu(B)\in\{0,1\}\},\qquad\mbox{and}\qquad R(\cdot)=\nu(\,\cdot\mid\mathcal{G}).\]
%% Obviously, $[x]_\mathcal{G}=\{x\}\in\mathcal{G}$, for $x\neq 0$, as $\nu(\{x\})=0$. Since all $\{[x]_\mathcal{G}\}_{x\in\mathbb{X}}$ form a partition of $\mathbb{X}$, then $[0]_\mathcal{G}=\{0\}$.
Then $\mathcal{G}$ is $\sigma$-algebra such that $[x]_\mathcal{G}=\{x\}\in\mathcal{X}$, for all $x\in\mathbb{X}$. Moreover, as measures on $\mathbb{X}$, we have $R_x(dy)=\nu(dy)$ for $\nu$-a.e. $x$, say on $C\in\mathcal{G}$ with $\nu(C)=1$; thus, $\sigma(R)\cap C=\{\emptyset,C\}$. It follows for every $A,B\in\mathcal{X}$ and all $x\in C$ that
%% Let $A\in\mathcal{G}$. If $\nu(A)=1$, then $\int_A\nu(B)\nu(dx)=\nu(B)=\nu(A\cap B)=\int_AR_x(B)\nu(dx)$. If $\nu(A)=0$, then $\int_A\nu(B)\nu(dx)=0=\nu(A\cap B)=\int_AR_x(B)\nu(dx)$.
\[\int_AR_y(B)R_x(dy)=\int_AR_y(B)\nu(dy)=\nu(A)\nu(B)=\int_BR_y(A)R_x(dy).\]
On the other hand, we have $R_0([0]_\mathcal{G})=\nu(\{0\})=1/2$, which implies that $R$ is not a.e. trivial w.r.t. $\mathcal{G}$. Notice, however, that $[0]_\mathcal{G}\notin\mathcal{G}$ and
\[R_x([x]_{\sigma(R)})=R_x(C)=\nu(C)=1,\qquad\mbox{for all }x\in C.\]
\end{example}

The following simple example relates to the second part of Theorem \ref{result:connections:reverse<->total}, showing that \textsf{(T)} w.r.t. $\sigma(R)$ alone is not enough to recover \textsf{(R)} (w.r.t. $\nu$).

\begin{example}\label{example:connections:three_colors}
Let $\mathbb{X}=\{1,2,3\}$. Suppose that $R$ is given in matrix form by
\[R=\left[\begin{array}{ccc} 1 & 0 & 0 \\ 0 & \frac{1}{2} & \frac{1}{2} \\ 0 & \frac{1}{2} & \frac{1}{2} \end{array}\right].\]
By Lemma \ref{lemma:connections:atoms_R}, the $\sigma(R)$-atoms are $\{1\}$ and $\{2,3\}$, so $R_x([x]_{\sigma(R)})=1$, for all $x\in\mathbb{X}$. On the other hand, it is easily seen that $R$ will not satisfy \textsf{(R)} unless $\nu$ is uniform on $\mathbb{X}$; in fact, $R$ will not satisfy \textsf{(S)} unless $\nu$ is uniform on $\mathbb{X}$.
\end{example}

Incidentally, Examples \ref{example:connections:improper} and \ref{example:connections:three_colors} help us understand the effects of \textsf{(T)} on $R$, and ultimately the structure of any r.c.d. By Lemma \ref{lemma:connections:atoms_R} and since $\{[x]_{\sigma(R)}\}_{x\in\mathbb{X}}$ forms a partition of $\mathbb{X}$, for all $x,y\in\mathbb{X}$, we have $[x]_{\sigma(R)}=[y]_{\sigma(R)}$ if and only if, as measures on $\mathbb{X}$, it holds $R_x(dz)=R_y(dz)$. Therefore, condition \textsf{(T)} w.r.t. $\sigma(R)$ imposes a ``block-diagonal'' design on $R$ in the sense that, for $\nu$-a.e. $x,y$ belonging to the same $\sigma(R)$-atom, the measures $R_x\equiv R_y$ are identical and give zero probability outside this atom. Theorem \ref{result:characterization:c.g.under<->total+reverse} shows, broadly speaking, that the additional assumption of \textsf{(S)} implies that $R_x$ is proportional to $\nu$ conditionally given $[x]_{\sigma(R)}$, which ultimately means that $R$ is an r.c.d. for $\nu$ given $\sigma(R)$.

Now, assume that $R(\cdot)=\nu(\,\cdot\mid\mathcal{G})$, for some sub-$\sigma$-algebra $\mathcal{G}\subseteq\mathcal{X}$. Since $\sigma(R)\subseteq\mathcal{G}$ and, by construction, $x\mapsto R_x(B)$ is $\sigma(R)$-measurable, for all $B\in\mathcal{X}$, we get $R(\cdot)=\nu(\,\cdot\mid\sigma(R))$. Moreover, as $\sigma(R)$ is c.g., Proposition \ref{result:introduction:proper:characterization} implies that $R$ satisfies \textsf{(T)} w.r.t. $\sigma(R)$. Therefore, $R$ and, consequently, any r.c.d. has the aforementioned ``block-diagonal'' design at the level of the $\sigma(R)$-atoms. Since $[x]_\mathcal{G}\subseteq[x]_{\sigma(R)}$, for all $x\in\mathbb{X}$, and $R_{x'}=R_x$, for all $x'\in[x]_{\sigma(R)}$, there is no additional structure emerging at the resolution of the $\mathcal{G}$-atoms, see also \cite[Section 29.2]{loeve1978}. The role of $\mathcal{G}$ and $\nu$ is to locate the $\sigma(R)$-atoms and determine the value of $R$ within each atom. On the other hand, the above observations can serve as a necessary criterion for a probability kernel to be an r.c.d.; e.g., in Example \ref{example:connections:countable}, the $\sigma(R)$-atoms are $\{1\}$, $\{2\}$, $\{3,4\}$, so $R$ will not be an r.c.d. for $\nu$ unless $\nu(\{1\})=0$.

\section{Characterization of a.e. proper r.c.d.s}\label{section:charcterization}

In this section, we demonstrate how certain classes of r.c.d.s are characterized by some combination of the properties stationarity, reversibility and totality. In particular, Theorem \ref{result:characterization:c.g.under<->total+reverse} states that a probability kernel $R$ on $\mathbb{X}$ which satisfies \textsf{(S)} and \textsf{(T)} w.r.t. $\sigma(R)$ is necessarily an r.c.d. for $\nu$ given some c.g. under $\nu$ sub-$\sigma$-algebra $\mathcal{G}\subseteq\mathcal{X}$, see also \cite[p.\,741]{blackwell1975}, \cite[Lemma 1]{berti2013} and \cite[Proposition 5.19]{einsielder2011}. Moreover, $\mathcal{G}$ is \textit{$\nu$-essentially unique} among the class of c.g. under $\nu$ sub-$\sigma$-algebras, in the sense that the difference between $\mathcal{G}$ and any other such conditioning $\sigma$-algebra is up to a set of $\nu$-probability zero. For the latter, we need the following preliminary lemma.

\begin{lemma}\label{lemma:characterization:uniqueness}
Let $R(\cdot)=\nu(\,\cdot\mid\mathcal{G})$. If $R$ satisfies \textnormal{\textsf{(P)}}, then there exists $C\in\mathcal{G}$ such that $\nu(C)=1$ and
\[\mathcal{G}\cap C=\sigma(R)\cap C.\]
\end{lemma}
\begin{proof}
The proof of this fact is essentially contained in the proof of Proposition \ref{result:introduction:proper:characterization}, part $(ii)\Rightarrow(i)$.
\end{proof} 

\begin{theorem}\label{result:characterization:c.g.under<->total+reverse}
Let $R$ be a probability kernel on $\mathbb{X}$. The following are equivalent:
\begin{enumerate}
\item[(i)] $R(\cdot)=\nu(\,\cdot\mid\mathcal{G})$, for some c.g. under $\nu$ sub-$\sigma$-algebra $\mathcal{G}\subseteq\mathcal{X}$;
\item[(ii)] $R$ satisfies \textnormal{\textsf{(S)}} and \textnormal{\textsf{(T)}} w.r.t. $\sigma(R)$.
\end{enumerate}
%$R(\cdot)=\nu(\,\cdot\mid\mathcal{G})$ for some c.g. under $\nu$ sub-$\sigma$-algebra $\mathcal{G}\subseteq\mathcal{X}$ if and only if $R$ satisfies \textnormal{\textsf{(S)}} and \textnormal{\textsf{(T)}} w.r.t. $\sigma(R)$.
Moreover, $\mathcal{G}$ is $\nu$-essentially unique among all sub-$\sigma$-algebras which are c.g. under $\nu$.
\end{theorem}
\begin{proof}
If $R(\cdot)=\nu(\,\cdot\mid\mathcal{G})$ and $\mathcal{G}$ is c.g. under $\nu$, then $R$ satisfies \textsf{(S)} and, by Proposition \ref{result:introduction:proper:characterization}, \textsf{(T)} w.r.t. $\mathcal{G}$. Since $\sigma(R)\subseteq\mathcal{G}$, we have $[x]_{\mathcal{G}}\subseteq[x]_{\sigma(R)}$ and $R_x([x]_{\sigma(R)})\geq R_x([x]_\mathcal{G})=1$, for $\nu$-a.e. $x$; thus, $R$ satisfies \textsf{(T)} w.r.t. $\sigma(R)$.

Conversely, suppose that $R$ satisfies \textsf{(S)} and \textsf{(T)} w.r.t. $\sigma(R)$ on $G\in\sigma(R)$, where $\nu(G)=1$. It follows from Section \ref{section:introduction} that $R$ is proper w.r.t. $\sigma(R)$ on $G$. Let $A\in\sigma(R)$ and $B\in\mathcal{X}$. Fix $x\in G$. If $x\in A$, then $R_x(A)=1$, so $R_x(A\cap B)=R_x(B)$; otherwise, if $x\in A^c$, then $R_x(A)=0$, so $R_x(A\cap B)=0$. Therefore, $R_x(A\cap B)=R_x(B)\delta_x(A)$. It now follows from \textsf{(S)} and $\nu(G)=1$ that
\[\int_AR_x(B)\nu(dx)=\int_\mathbb{X}R_x(B)\delta_x(A)\nu(dx)=\int_GR_x(A\cap B)\nu(dx)=\nu(A\cap B).\]
By construction, $x\mapsto R_x(B)$ is $\sigma(R)$-measurable, for all $B\in\mathcal{X}$, so $\nu(\,\cdot\mid\sigma(R))$, with $\sigma(R)$ being c.g. as required.

Regarding the last statement, suppose that $R(\cdot)=\nu(\,\cdot\mid\mathcal{H})$, where $\mathcal{H}$ is another c.g. under $\nu$ sub-$\sigma$-algebra of $\mathcal{X}$ on $\mathbb{X}$. By Lemma \ref{lemma:characterization:uniqueness}, there exist $C_1,C_2\in\mathcal{X}$ such that $\nu(C_1)=1$, $\nu(C_2)=1$, $\mathcal{G}\cap C_1=\sigma(R)\cap C_1$ and $\mathcal{H}\cap C_2=\sigma(R)\cap C_2$. Therefore,
\[\mathcal{G}\cap C_1\cap C_2=\sigma(R)\cap C_1\cap C_2=\mathcal{H}\cap C_1\cap C_2.\]
\end{proof}

One consequence of Theorem \ref{result:characterization:c.g.under<->total+reverse}, related to the discussion in Example \ref{example:connections:three_colors}, is that \textsf{(T)} w.r.t. $\sigma(R)$ together with \textsf{(S)} implies \textsf{(R)}. On the other hand, by Theorem \ref{result:connections:reverse<->total}, any $R$ satisfying \textsf{(S)} and \textsf{(SR)} is an r.c.d. for $\nu$ given some c.g. under $\nu$ sub-$\sigma$-algebra.

\begin{remark}\label{remark:characterization:c.g.under<->total+reverse:nu}
%% Or show immediately that $R(\cdot)=\nu_{\sigma(R)}(\cdot\mid\sigma(R))$
If our goal is to show that $R$ is an r.c.d. for \textit{some} probability measure on $\mathbb{X}$, as is the case in \cite{berti2013,preston1976}, then it is possible to drop condition \textsf{(S)} from Theorem \ref{result:characterization:c.g.under<->total+reverse} altogether. Indeed, suppose that $R$ satisfies only \textsf{(T)} w.r.t. $\sigma(R)$ (and $\nu$). Let us define the probability measure
\[\pi(dy):=\int_\mathbb{X}R_x(dy)\nu(dx).\]
It follows for every $A\in\sigma(R)$ that $\pi(A)=\int_\mathbb{X}R_x(A)\nu(dx)=\int_\mathbb{X}\delta_x(A)\nu(dx)=\nu(A)$; thus, $\pi_{\sigma(R)}\equiv\nu_{\sigma(R)}$, and so $R$ satisfies \textsf{(T)} w.r.t. $\sigma(R)$ and $\pi$. Moreover, we have from Lemma \ref{lemma:connections:atoms_R} that, for every $B\in\mathcal{X}$,
\[\int_\mathbb{X}R_x(B)\pi(dx)=\int_{\mathbb{X}^2}R_x(B)R_v(dx)\nu(dv)=\int_\mathbb{X}R_x(B)R_x(\mathbb{X})\nu(dx)=\int_\mathbb{X}R_x(B)\nu(dx)=\pi(B).\]
%% Or use the argument from Remark \ref{remark:characterization:c.g.under<->total+reverse:sigmaR}.
Therefore, from the proof of Theorem \ref{result:characterization:c.g.under<->total+reverse}, we get $R(\cdot)=\pi(\,\cdot\mid\sigma(R))$.
\end{remark}

R.c.d.s for which the conditioning $\sigma$-algebra $\mathcal{G}=\sigma(G_1,G_2,\ldots)$ is generated by a countable partition $\{G_1,G_2,\ldots\}$ of $\mathbb{X}$ form an important class. In that case, $\{G_1,G_2,\ldots\}$ make up the collection of $\mathcal{G}$-atoms. Furthermore, using a monotone class argument and the fact that $\mathcal{X}$ is c.g., we can show that, as measures on $\mathbb{X}$,
\begin{equation}\label{eq:characterization:partition}
\nu(\,\cdot\mid\mathcal{G})(x)=\sum_m\nu(\,\cdot\mid G_m)\cdot\mathbbm{1}_{G_m}(x)=\nu(\,\cdot\mid[x]_\mathcal{G}),\qquad\mbox{for }\nu\mbox{-a.e. }x,
\end{equation}
where we let $\nu(\,\cdot\mid G_m)=1$ whenever $\nu(G_m)=0$. Therefore, conditioning on $\mathcal{G}$ here means conditioning on the $\mathcal{G}$-atom that actually occurs, so in particular,
\[\nu([x]_\mathcal{G}|\mathcal{G})(x)=1\qquad\mbox{for }\nu\mbox{-a.e. }x.\]
On the other hand, it is not hard to see from \eqref{eq:characterization:partition} that $\nu(\,\cdot\mid\mathcal{G})(x)$ is absolutely continuous w.r.t. $\nu$, for $\nu$-a.e. $x$, which we will denote by $\nu(\,\cdot\mid\mathcal{G})(x)\ll\nu$.

The next result is a concretization of Theorem \ref{result:characterization:c.g.under<->total+reverse} and characterizes all r.c.d.s for which the conditioning $\sigma$-algebra $\mathcal{G}$ is \textit{generated by a countable partition (g.c.p.) under $\nu$}, in the sense that $\mathcal{G}\cap C=\sigma(D_1,D_2,\ldots)$ for some $C,D_1,D_2,\ldots\in\mathcal{G}$ such that $\nu(C)=1$ and $\{D_1,D_2,\ldots\}$ is a partition of $C$. In that case, it still holds
\[\nu(\,\cdot\mid\mathcal{G})(x)=\sum_n\nu(\,\cdot\mid D_n)\cdot\mathbbm{1}_{D_n}(x)\qquad\mbox{for }\nu\mbox{-a.e. }x.\]

\begin{theorem}\label{result:characterization:c.g.under<->total+reverse:dominated}
Let $R$ be a probability kernel on $\mathbb{X}$. The following are equivalent:
\begin{enumerate}
\item[(i)] $R(\cdot)=\nu(\,\cdot\mid\mathcal{G})$, for some g.c.p. under $\nu$ sub-$\sigma$-algebra $\mathcal{G}\subseteq\mathcal{X}$;
\item[(ii)] $R$ satisfies \textnormal{\textsf{(S)}}, \textnormal{\textsf{(T)}} w.r.t. $\sigma(R)$, and $R_x\ll\nu$ for $\nu$-a.e. $x$.
\end{enumerate}
%$R(\cdot)=\nu(\,\cdot\mid\mathcal{G})$ for some sub-$\sigma$-algebra $\mathcal{G}\subseteq\mathcal{X}$, which is generated by a countable partition under $\nu$, if and only if $R$ satisfies \textnormal{\textsf{(S)}}, \textnormal{\textsf{(T)}} w.r.t. $\sigma(R)$, and $R_x\ll\nu$ for $\nu$-a.e. $x$.
Moreover, $\mathcal{G}$ is $\nu$-essentially unique among all sub-$\sigma$-algebras which are g.c.p. under $\nu$.
\end{theorem}
\begin{proof}
We have already shown the necessity of \textsf{(S)}, \textsf{(T)} w.r.t. $\sigma(R)$, and $R_x\ll\nu$ for $\nu$-a.e. $x$. Conversely, suppose that $R$ satisfies \textsf{(S)}, \textnormal{\textsf{(T)}} w.r.t. $\sigma(R)$, and $R_x\ll\nu$ for $\nu$-a.e. $x$. Theorem \ref{result:characterization:c.g.under<->total+reverse} and Proposition \ref{result:introduction:proper:characterization} imply that $(i)$ $R=\nu(\,\cdot\mid\mathcal{G})$ for some sub-$\sigma$-algebra $\mathcal{G}$ of $\mathcal{X}$, and $(ii)$ $R$ is a.e. proper w.r.t. $\mathcal{G}$, say on $F\in\mathcal{G}$ with $\nu(F)=1$. Moreover, from the properties of r.c.d.s, $R$ satisfies \textsf{(R)} and \textsf{(SR)}. It now follows from the proof of Theorem 3.10 in \cite{sariev2024} that there exists a partition $\{D_1,D_2,\ldots\}$ of $\mathbb{X}$ in $\mathcal{X}$ such that, letting $\mathcal{H}=\sigma(D_1,D_2\ldots)$, it holds, as measures on $\mathbb{X}$,
\begin{equation}\label{example:representation:integral_equations:dominated:eq1}
R_x(dy)=\nu(dy|\mathcal{H})(x)\qquad\mbox{for }\nu\mbox{-a.e. }x.
\end{equation}
%% otherwise, $R(\cdot)=\nu(\,\cdot\mathcal{H}')$ on the complete measure space $(\mathbb{X},\mathcal{X}',\nu)$, where $\mathcal{H}'$ is the completion of $\mathcal{H}$
Take $C\in\mathcal{X}$, with $\nu(C)=1$, to be simultaneously the almost sure set in \eqref{example:representation:integral_equations:dominated:eq1} and such that $\nu(H|\mathcal{H})(x)=\delta_x(H)$, for all $H\in\mathcal{H}$ and $x\in C$. It is then straightforward to check that
\[\mathcal{G}\cap C\cap F=\mathcal{H}\cap C\cap F.\]
Finally, define $\mathcal{G}':=\{(D\cap C)\cup(E\cap C^c):D,E\in\mathcal{G}\}$. Then $\mathcal{G}'$ is a $\sigma$-algebra such that $\mathcal{G}\subseteq\mathcal{G}'$, $C\in\mathcal{G}'$, and $\mathcal{G}'\cap C\cap F=\mathcal{H}\cap C\cap F$; thus, $\mathcal{G}'$ is g.c.p. under $\nu$. Moreover, it is not hard to see that $R(\cdot)=\nu(\,\cdot\mid\mathcal{G}')$. The uniqueness part follows as in the proof of Theorem \ref{result:characterization:c.g.under<->total+reverse}.
\end{proof}

\section{More on r.c.d.s}\label{section:proper}

Intuition tells us that when we condition on a $\sigma$-algebra, we actually expect to be conditioning on the particular atom that occurs. Indeed, when the sub-$\sigma$-algebra $\mathcal{G}\subseteq\mathcal{X}$ is g.c.p., we have seen that the probability measure $\nu(\,\cdot\mid\mathcal{G})(x)$ coincides with $\nu(\,\cdot\mid[x]_\mathcal{G})$, for $\nu$-a.e. $x$. Then, in particular, $\nu(A|\mathcal{G})(x)=1$ whenever $x\in A\in\mathcal{G}$, which \cite{blackwell1975} calls an ``intuitive desideratum'' for r.c.d.s. However, for general $\mathcal{G}$, the $\mathcal{G}$-atoms can be uncountably many, so most if not all of them will be $\nu$-null sets (provided they are even measurable). In that case, $\nu(\,\cdot\mid[x]_\mathcal{G})$ is left undefined and it may even be that $\nu(\,\cdot\mid\mathcal{G})(x)$ and $\nu$ are mutually singular. For example, if $\mathcal{G}$ is c.g. under $\nu$, then Proposition \ref{result:introduction:proper:characterization} implies for $\nu$-a.e. $x$ that $\nu([x]_\mathcal{G}|\mathcal{G})(x)=1$, even if $\nu([x]_\mathcal{G})=0$. Moreover, as Proposition \ref{result:introduction:proper:characterization} suggests, properness of $\nu(\,\cdot\mid\mathcal{G})$ depends on the particular $\mathcal{G}$ which is used. Another problem that arises, and which is highlighted by the following remark, is that when we condition on a set, which is simultaneously an atom of two different sub-$\sigma$-algebras, the resulting probability measure may depend on the particular $\sigma$-algebra used, see also \citep[][Section 2.1]{seidenfeld2001}. Therefore, conditioning on $\sigma$-algebras is, in general, conceptually different from conditioning on events.

\begin{remark}[Borel-Kolmogorov paradox]\label{remark:proper:borel-kolmogorov}
Let $X$ and $Y$ be random variables on a probability space $(\Omega,\mathcal{F},\mathbb{P})$. Then $[x]_{\sigma(X)}=\{X=x\}$ and $[y]_{\sigma(Y)}=\{Y=y\}$, for all $x,y\in\mathbb{R}$. Suppose that $A=\{X=x^*\}=\{Y=y^*\}$ for some $x^*,y^*\in\mathbb{R}$. Take $\omega\in A$. If $X$ and $Y$ are different, then it may happen that $\mathbb{P}(\,\cdot\mid X)(\omega)\neq\mathbb{P}(\,\cdot\mid Y)(\omega)$, see, e.g., \citep[][Problem 33.1]{billingsley1995}.
\end{remark}

Proposition \ref{result:proper:r.c.d.s} collects several facts that are consistent with our intuition about r.c.d.s. First, it reinforces the idea that, conditionally on $\mathcal{G}$, we observe the atom $[x]_\mathcal{G}$, not the point $x$. Thus, we expect two different points $x,y$ to provide the same information when $[x]_\mathcal{G}=[y]_\mathcal{G}$, see \cite[p.\,7]{berti2020}. Moreover, Proposition \ref{result:proper:r.c.d.s} shows that r.c.d.s agree with the usual notion of conditional probabilities when conditioning on atoms with positive $\nu$-probability. On the other hand, the Borel-Kolmogorov paradox, as stated, is a zero-probability problem, which will not be true if enough atoms are shared between the conditioning $\sigma$-algebras. However, in the case of c.g. under $\nu$ sub-$\sigma$-algebras sharing a.e. of their atoms, we show that the conditional distributions coincide at $\nu$-a.e. point, see also \cite[Theorem 5.14(4)]{einsielder2011}.
%% In \cite[Theorem 5.14(4)]{einsielder2011}, if $\mathcal{G}\overset{=}{\nu}\mathcal{H}$, then $\nu(\,\cdot\mid\mathcal{G})(x)=\nu(\,\cdot\mid\mathcal{H})(x)$ for $\nu$-a.e. $x$.

\begin{proposition}\label{result:proper:r.c.d.s}
Let $\mathcal{G}$ and $\mathcal{H}$ be sub-$\sigma$-algebras of $\mathcal{X}$. Then
\begin{enumerate}
\item $\nu(\,\cdot\mid\mathcal{G})(x)=\nu(\,\cdot\mid\mathcal{G})(y)$ when $[x]_\mathcal{G}=[y]_\mathcal{G}$, for all $x,y\in\mathbb{X}$;
\item $\nu(\,\cdot\mid\mathcal{G})(x)=\nu(\,\cdot\mid G)$ for $\nu$-a.e. $x\in G$, where $G\in\mathcal{G}$ is a $\mathcal{G}$-atom such that $\nu(G)>0$. 
\end{enumerate}
If, in addition, $\mathcal{G},\mathcal{H}$ are both c.g. under $\nu$ and $\nu$-a.e. of their atoms coincide, then
\begin{enumerate}
\item[(iii)] $\nu(\,\cdot\mid\mathcal{G})(x)=\nu(\,\cdot\mid\mathcal{H})(x)$ for $\nu$-a.e. $x$.
\end{enumerate}
\end{proposition}
\begin{proof}
The first result follows from Lemma \ref{lemma:connections:constant}.
%Regarding $(i)$, fix $x\in\mathbb{X}$. Let $x'\in[x]_\mathcal{G}$ and $B\in\mathcal{X}$. Define $D_B:=\{y\in\mathbb{X}:\nu(B|\mathcal{G})(x)=\nu(B|\mathcal{G})(y)\}$. Then $D_B\in\mathcal{G}$ and $x\in D_B$, which implies that $x'\in D_B$, i.e. $\nu(B|\mathcal{G})(x)=\nu(B|\mathcal{G})(x')$.
Regarding $(ii)$, let $B\in\mathcal{X}$. Take $A\in\mathcal{G}$ such that $A\cap G\neq\emptyset$. Since $G$ is a $\mathcal{G}$-atom, we have $A\cap G=G$, and so
\begingroup\allowdisplaybreaks
\begin{align*}
\int_A\nu(B|\mathcal{G})(x)\delta_x(G)\nu(dx)&=\nu(B\cap G)=\int_G\nu(B|G)\nu(dx)=\int_A\nu(B|G)\delta_x(G)\nu(dx).
\end{align*}
\endgroup
The same is trivially true if we assumed $A\cap G=\emptyset$ instead. Since $x\mapsto\nu(B|\mathcal{G})(x)\delta_x(G)$ is $\mathcal{G}$-measurable and $\mathcal{X}$ is c.g., combining the essential sets, we obtain $(ii)$.

Regarding $(iii)$, by hypothesis, there exists $C\in\mathcal{X}$ such that $\nu(C)=1$, $[x]_\mathcal{G}=[x]_\mathcal{H}$ for all $x\in C$, and $\mathcal{G}\cap C$ and $\mathcal{H}\cap C$ are c.g. It follows that $[x]_{\mathcal{G}\cap C}=[x]_\mathcal{G}\cap C=[x]_\mathcal{H}\cap C=[x]_{\mathcal{H}\cap C}$. Therefore, $\mathcal{G}\cap C$ and $\mathcal{H}\cap C$ have the same atoms, so \cite[Corollary 1]{blackwell1956} implies that $\mathcal{G}\cap C=\mathcal{H}\cap C$. Let $A\in\mathcal{G}\cap\mathcal{H}$ and $B\in\mathcal{X}$. Then
\[\int_A\nu(B|\mathcal{G})(x)\delta_x(C)\nu(dx)=\nu(A\cap B)=\int_A\nu(B|\mathcal{H})(x)\delta_x(C)\nu(dx).\]
Since $\nu(B|\mathcal{G})\delta(C)$ and $\nu(B|\mathcal{H})\delta(C)$ are $\mathcal{G}\cap\mathcal{H}\cap C$-measurable, and $\mathcal{X}$ is c.g., we obtain, as measures on $\mathbb{X}$, that $\nu(dy|\mathcal{G})(x)\delta_x(C)=\nu(dy|\mathcal{H})(x)\delta_x(C)$ for $\nu$-a.e. $x$.
\end{proof}

\subsection*{Acknowledgments}

This study is financed by the European Union-NextGenerationEU, through the National Recovery and Resilience Plan of the Republic of Bulgaria, project No. BG-RRP-2.004-0008.

\bibliography{mybib} 

\end{document}